%% file: technical_report_arxiv.tex
\documentclass[12pt]{article}
\usepackage{amsmath,amsfonts,amsthm,amssymb}
\usepackage[square]{natbib}
\usepackage[utf8]{inputenc}
\usepackage[english]{babel}
\usepackage[nottoc,notlof,notlot]{tocbibind}
\usepackage{microtype,ulem}
\usepackage{hyperref}
\usepackage{xcolor,tikz}
\usetikzlibrary{graphs,arrows}
\hypersetup{
    colorlinks,
    linkcolor={red!50!black},
    citecolor={blue!50!black},
    urlcolor={blue!80!black}
}
\usepackage[inner=2.6cm,outer=1.8cm]{geometry}
\input{defs.tex}
\allowdisplaybreaks

\title{Exponential decay of scattering coefficients}
\author{Irène Waldspurger
\footnote{Previously at département d'informatique de l'ENS Paris; now at MIT Institute for Data, Systems and Society.}}
\date{}

\begin{document}

\maketitle

This technical report is the fifth chapter of \citep{these}.

\section{Introduction}

When trying to analyze or interpret a complex signal (image, audio signal, temporal series...), it is often difficult to directly work with the natural representation of the signal (an array of pixels, in the case of an image). On the contrary, using a more appropriate representation can render the task much easier to solve. The scattering transform, introduced by \citet{group_invariant}, can serve as such a representation for various categories of signals and tasks.

Defined over $L^2(\R^d)$, the scattering transform consists in a cascade of wavelet transform modulus. It has the property to be invariant to translations, and stable to small deformations. In the last years, it has proven to be a very efficient tool in diverse settings \citep{anden,oyallon,hirn}.

It belongs to the family of deep representations, but, contrarily to most other members of this family, it is not learned from data. Its theoretical properties are therefore easier to study, and shed some light on the behaviour of deep representations in general. In particular, through the scattering transform, we can hope to approach the following general question: what information contained in the original signal does a deep transform capture? This question is central for the understanding of all deep representations, but for learned ones, only empirical studies have until now been possible \citep{dosovitskiy, simonyan_deep_inside, mahendran}.

For the scattering transform, \citet{bruna_intermittent} have already shown that it characterizes important parameters of some classes of stationary processes. In the present work, we consider the scattering transform over $L^2(\R)$, and partially describe the amount of information contained in each layer of the transform. Informally, we show that, for some $r>0,a>1$, the $n$-th layer of the scattering transform is almost insensitive to the frequencies of the signal that belong to the range $[-ra^n;ra^n]$.

The exact statement of this result is given in Theorem \ref{thm:exp_decay}. Some hypotheses on the wavelets are necessary, but there are much less restrictive than the ones used in \citep{group_invariant}. For simplicity, we restrict ourselves to the one-dimensional case, but we believe that the result could be extended to $L^2(\R^d)$ for $d\geq 1$.

The main two consequences of this result are the following.
\begin{itemize}
\item It generalizes the results of \citep[Section 2]{group_invariant}: it allows to relax the \textit{admissibility condition} on the wavelets.
\item It explains why, in applications, high order scattering coefficients can be neglected. In cases where higher non-linearity orders seem desirable, this should incite us to redesign the definition of high order layers.
\end{itemize}

\nl
The organization is as follows. In Section~\ref{s:scattering_transform}, we precisely define the scattering transform and some known results about it. We also describe in more detail the propagation phenomenon. In Section~\ref{s:exp_decay_thm}, we state the theorem. In Section~\ref{s:exp_decay_proof}, we give the principle of the proof. Section~\ref{s:exp_decay_adaptation} adapts the theorem to the scattering transform on stationary processes. Finally, Section~\ref{s:exp_decay_technical} proves the lemmas necessary for the proof of the main theorem.

\section{The scattering transform\label{s:scattering_transform}}

As in the previous chapters, once a wavelet $\psi\in L^1\cap L^2(\R)$ (such that $\int_\R\psi=0$) has been chosen, we define a family of wavelets $(\psi_j)_{j\in\Z}$ by:
\begin{align*}
&\forall j\in\Z,t\in\R\quad\quad \psi_j(t)=2^{-j}\psi(2^{-j}t)\\
\iff\quad&\forall j\in\Z,\omega\in\R\quad\quad\hat\psi_j(\omega)=\hat\psi(2^j\omega)
\end{align*}

\subsection{Definition}

We follow the definition of \citep{group_invariant}.

The scattering transform consists in a cascade of modulus of wavelet transforms. After each application of the modulus, the resulting functions are locally averaged. The set of averages constitutes the scattering transform.

The averaging is performed with a real-valued positive function $\phi\in L^1\cap L^2(\R)$ such that $\hat\phi(0)=1$. We define:
\begin{equation*}
\forall J\in\Z,t\in\R\quad\quad \phi_J(t)=2^{-J}\phi(2^{-J}t)
\end{equation*}
The convolution with $\phi_J$ represents an average on an interval of characteristic size $2^J$.

\nl
We now formally define the cascade of modulus of wavelet transforms. For any function $f\in L^2(\R)$, we set:
\begin{equation*}
U[\o]f=f
\end{equation*}
and iteratively define, for any $n$-uplet $(j_1,...,j_n)\in\Z^n$, with $n\geq 1$:
\begin{equation*}
U[(j_1,...,j_n)]f=\big|U[(j_1,...,j_{n-1})]f\star \psi_{j_n}\big|
\end{equation*}

\nl
For any $J\in\Z$, we set $\mathcal{P}_J=\left\{(j_1,...,j_n), n\in\N,j_1,...,j_n\in\{-\infty,...,J\}\right\}$; we refer to the elements of $\mathcal{P}_J$ as \textit{paths}. We denote the length (that is, the number of elements) of a path $p$ by $|p|$.

For any $p\in\mathcal{P}_J$, we define:
\begin{equation*}
S_J[p]f=U[p]f\star\phi_J
\end{equation*}
The \textit{scattering coefficients} associated to $f$ at scale $J$ are the set $\{S_J[p]f\}_{p\in\mathcal{P}_J}$.

\nl
The computation of the scattering coefficients is schematized in Figure~\ref{fig:scattering}.

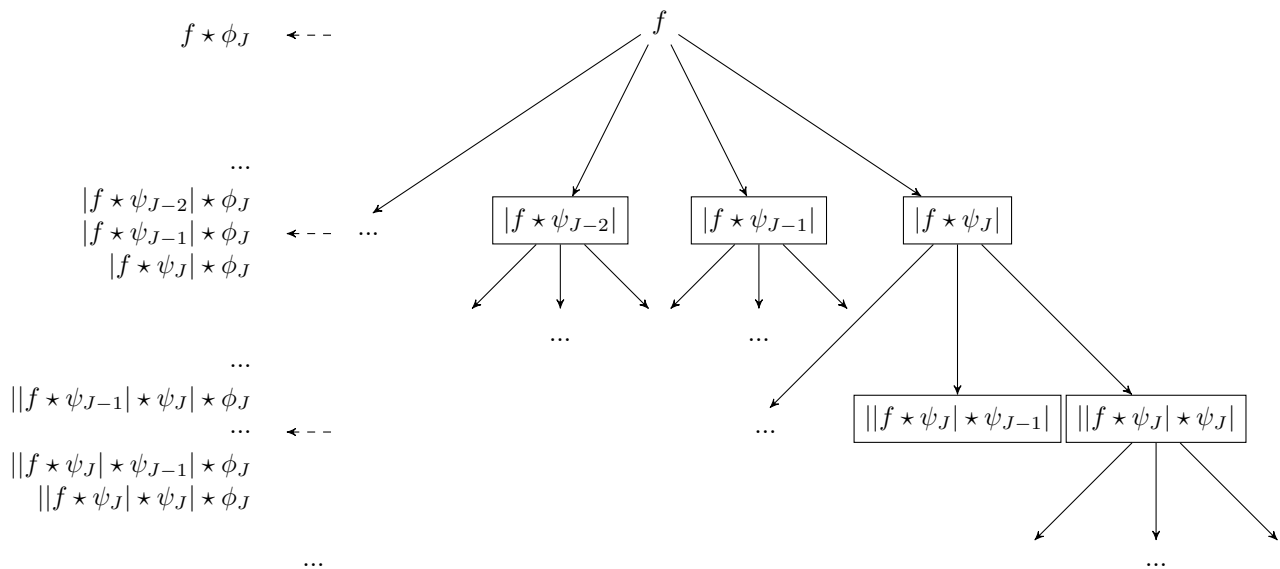
\begin{figure}
\centering
\input{tikz_scattering.tex}
\caption{Schematic illustration of the scattering transform: the tree on the right represents the cascade of modulus of wavelet transforms; the output scattering coefficients are on the left.\label{fig:scattering}}
\end{figure}

\subsection{Norm preservation and energy propagation}

When the wavelets are suitably chosen, the scattering transform preserves the norm \citep[Theorem 2.6]{group_invariant}:
\begin{equation}\label{eq:scattering_norm_preservation}
\underset{p\in\mathcal{P}_J}{\sum}||S_J[p]f||_2^2=||f||_2^2
\end{equation}
In this paragraph, we briefly describe this result. To explain it, we introduce the informal idea that the repeated application of the modulus of the wavelet transform moves the energy contained in the high frequencies of the initial signal towards the low frequencies.

\nl
For the moment, we assume the wavelets to be analytical:
\begin{equation*}
\hat\psi(\omega)=0\mbox{ if }\omega<0
\end{equation*}
and, together with the low-pass filter $\phi$, to satisfy a Littlewood-Paley condition:
\begin{equation*}
\forall \omega\geq 0\quad\quad
|\hat\phi(\omega)|^2+\frac{1}{2}\underset{j\leq 0}{\sum}|\hat\psi_j(\omega)|^2=1
\end{equation*}
This condition ensures that, for any scale $J\in\Z$, the wavelet transform with low-pass filter $f\to\left\{f\star\phi_J,\{f\star\psi_j\}_{j\leq J}\right\}$ is unitary over the set of real-valued functions:
\begin{equation*}
\forall f\in L^2(\R,\R)\quad\quad
||f\star\phi_J||_2^2+\underset{j\leq J}{\sum}||f\star\psi_j||^2_2=||f||_2^2
\end{equation*}
Consequently, the application $W_J:f\to\left\{f\star\phi_J,\{|f\star\psi_j|\}_{j\leq J}\right\}$ preserves the norm. As the scattering transform is computed by recursively applying this operator to the input function, we can prove by iteration over $n$ that, for any length $n\geq 0$:
\begin{equation*}
\forall f\in L^2(\R,\R)\quad\quad
\underset{p\in\mathcal{P}_J,|p|< n}{\sum}||S_J[p]f||_2^2 \,+
\underset{p\in\mathcal{P}_J,|p|= n}{\sum}||U[p]f||_2^2
= ||f||_2^2
\end{equation*}
If we can prove that $\sum_{p\in\mathcal{P}_J,|p|= n}||U[p]f||_2^2$ goes to zero when $n$ goes to infinity, then we can prove Equation~\eqref{eq:scattering_norm_preservation}. Under an additional condition on the wavelets, this is done in \citep[Theorem 2.6]{group_invariant}:
\begin{thm}\label{thm:norm_preservation}
A family $(\phi,\{\psi_j\}_{j\in\Z})$ is said to be admissible if there exists $\eta\in\R$ and a positive real-valued function $\rho\in L^2(\R)$ such that:
\begin{equation*}
\forall\omega\in\R,\quad|\hat\rho(\omega)|\leq |\hat\phi(2\omega)|
\quad\quad\mbox{and}\quad\quad
\hat\rho(0)=1
\end{equation*}
and that the function:
\begin{equation*}
\hat\Psi(\omega)=|\hat\rho(\omega-\eta)|^2-\underset{k=1}{\overset{+\infty}{\sum}}k(1-|\hat\rho(2^{-k}(\omega-\eta))|^2)
\end{equation*}
satisfies:
\begin{equation}\label{eq:admissibility_condition}
\underset{1\leq\omega\leq 2}{\inf}
\sum_{j=-\infty}^{+\infty}\hat\Psi(2^j\omega)|\hat\psi_j(\omega)|^2>0
\end{equation}
If $(\phi,\{\psi_j\}_{j\in\Z})$ is admissible, then, for any real-valued function $f\in L^2(\R)$ and any scale $J\in\Z$:
\begin{equation}\label{eq:scattering_to_zero}
\underset{p\in\mathcal{P}_J,|p|=n}{\sum}||U[p]f||_2^2\to 0\quad\quad
\mbox{when }n\to+\infty
\end{equation}
which implies the norm preservation:
\begin{equation*}
\underset{p\in\mathcal{P}_J}{\sum}||S_J[p]f||_2^2=||f||_2^2
\end{equation*}
\end{thm}

Intuitively, the property~\eqref{eq:scattering_to_zero} holds because the iterative application of the modulus of the wavelet transform moves the energy carried by the high frequencies of the signal towards the low frequencies. At each step of the scattering transform, the energy in the lowest frequency bands is output by convolution with $\phi_J$. The remaining part is shifted towards the lower frequencies by a new application of the modulus of the wavelet transform and so on.

For the displacement of the energy towards the low frequencies, the modulus is essential: for each signal $f$, $f\star\psi_j$ is a function whose energy is concentrated in a frequency band of characteristic size $2^{-j}$, with a mean frequency also of the order of $2^{-j}$. After application of the modulus, $|f\star\psi_j|$ tends to have its energy concentrated in a frequency band of characteristic size still equal to $2^{-j}$, but now centered around zero. So the frequencies that it contains are globally lower than the frequencies of $f\star\psi_j$. An example of this phenomenon is displayed on Figure~\ref{fig:shift_low}.

According to this simplistic reasoning, the modulus of the wavelet transform approximately moves the energy contained in a frequency band around $2^j$ to the frequency band $[-2^{j-1},2^{j-1}]$. By iteratively applying this argument, we expect the energy to arrive in the frequency band $[-2^J,2^J]$ (and thus disappear) after a number of scattering steps proportional to $j$. The theorem of the next section will formalize this idea.

\begin{figure}
\centering
\begin{tabular}{cccc}
\includegraphics[width=0.22\textwidth]{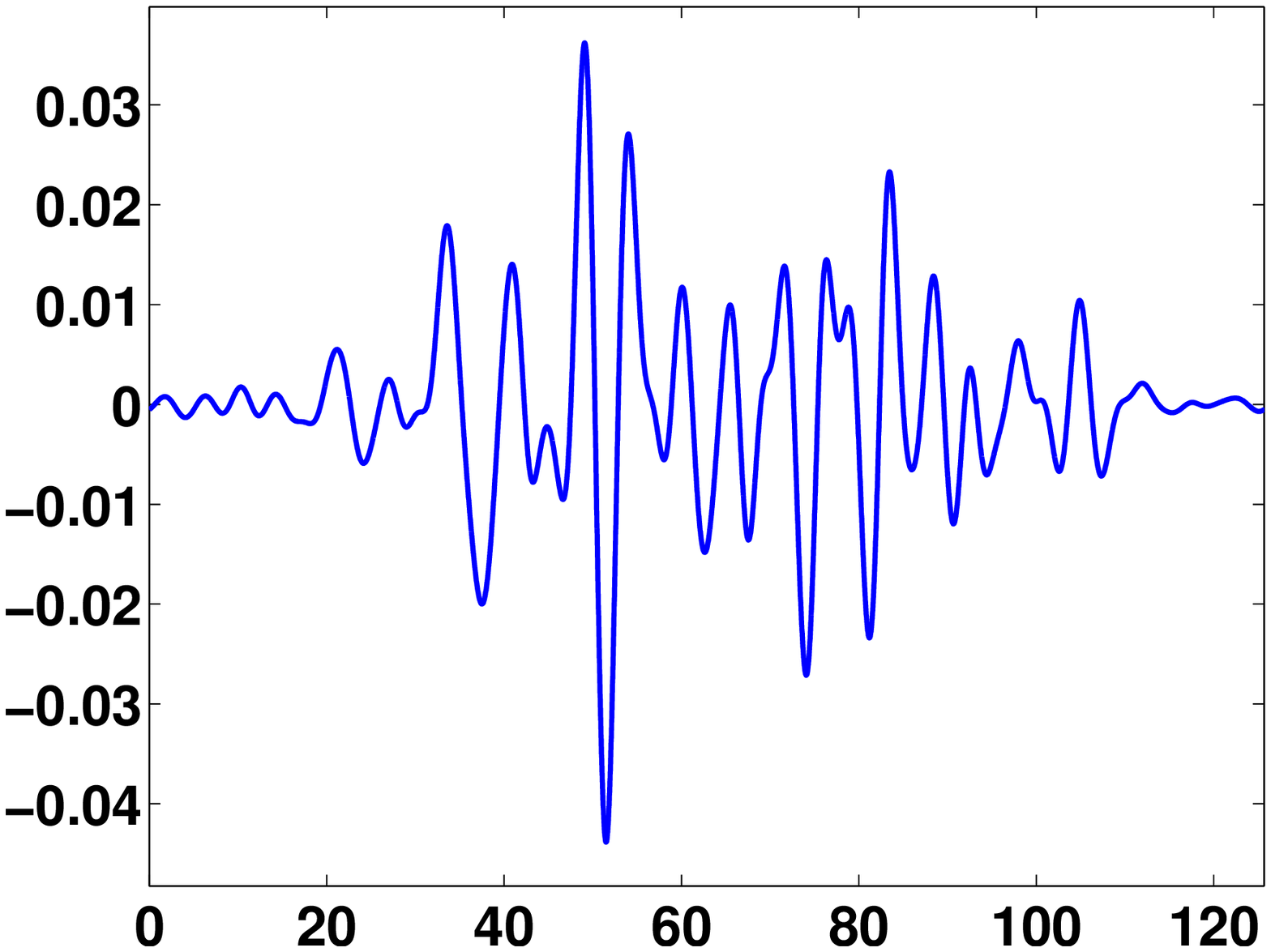}&
\includegraphics[width=0.22\textwidth]{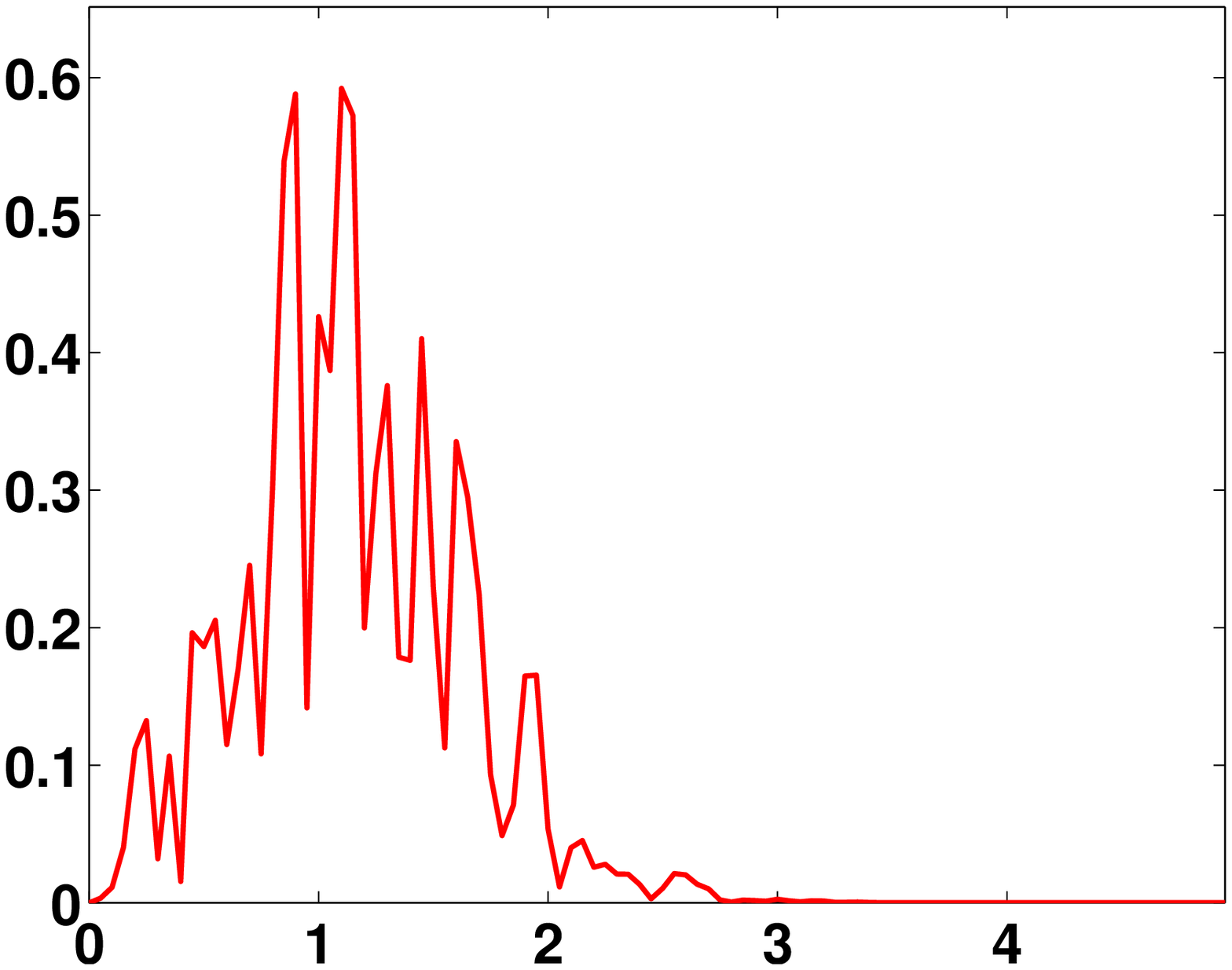}&
\includegraphics[width=0.22\textwidth]{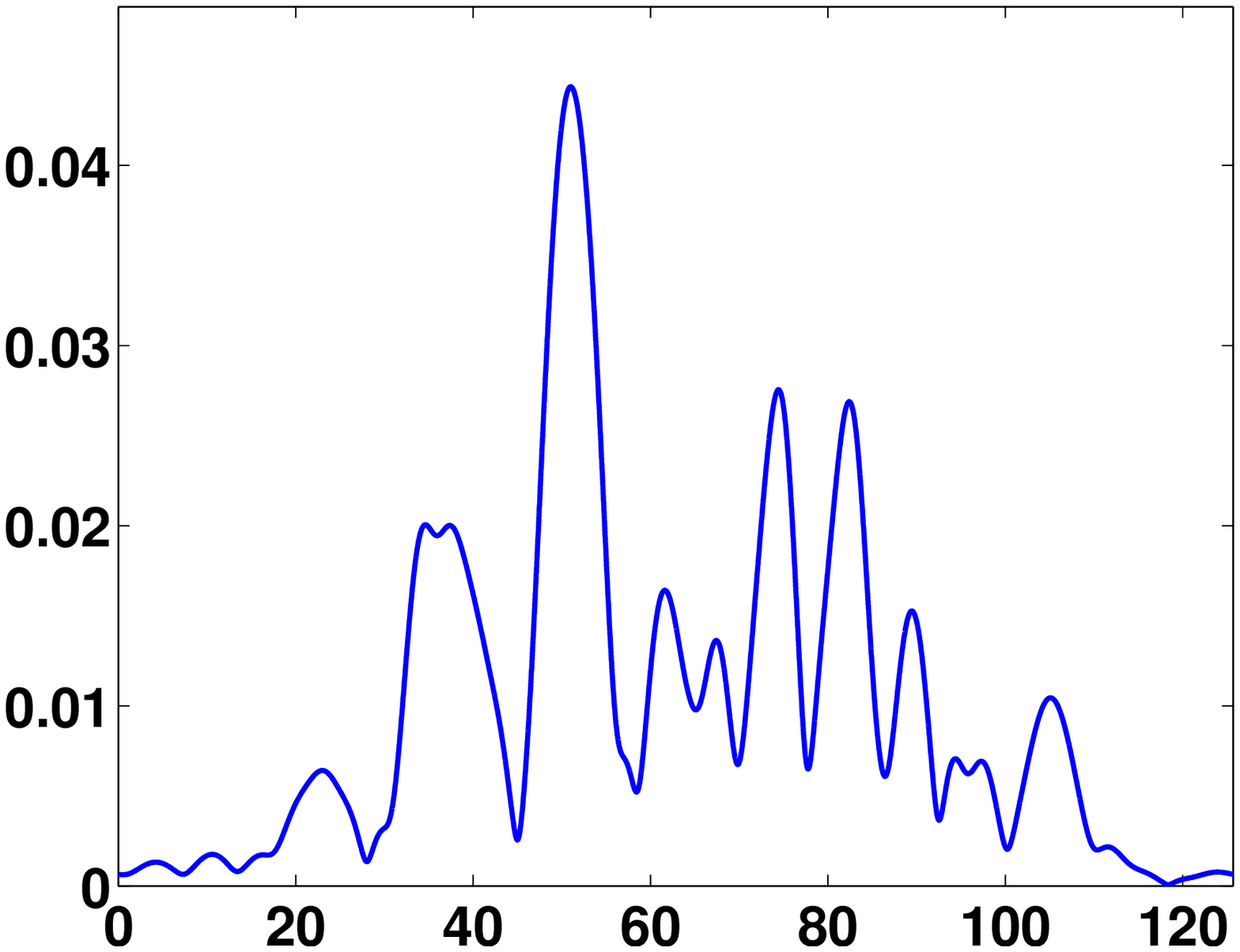}&
\includegraphics[width=0.22\textwidth]{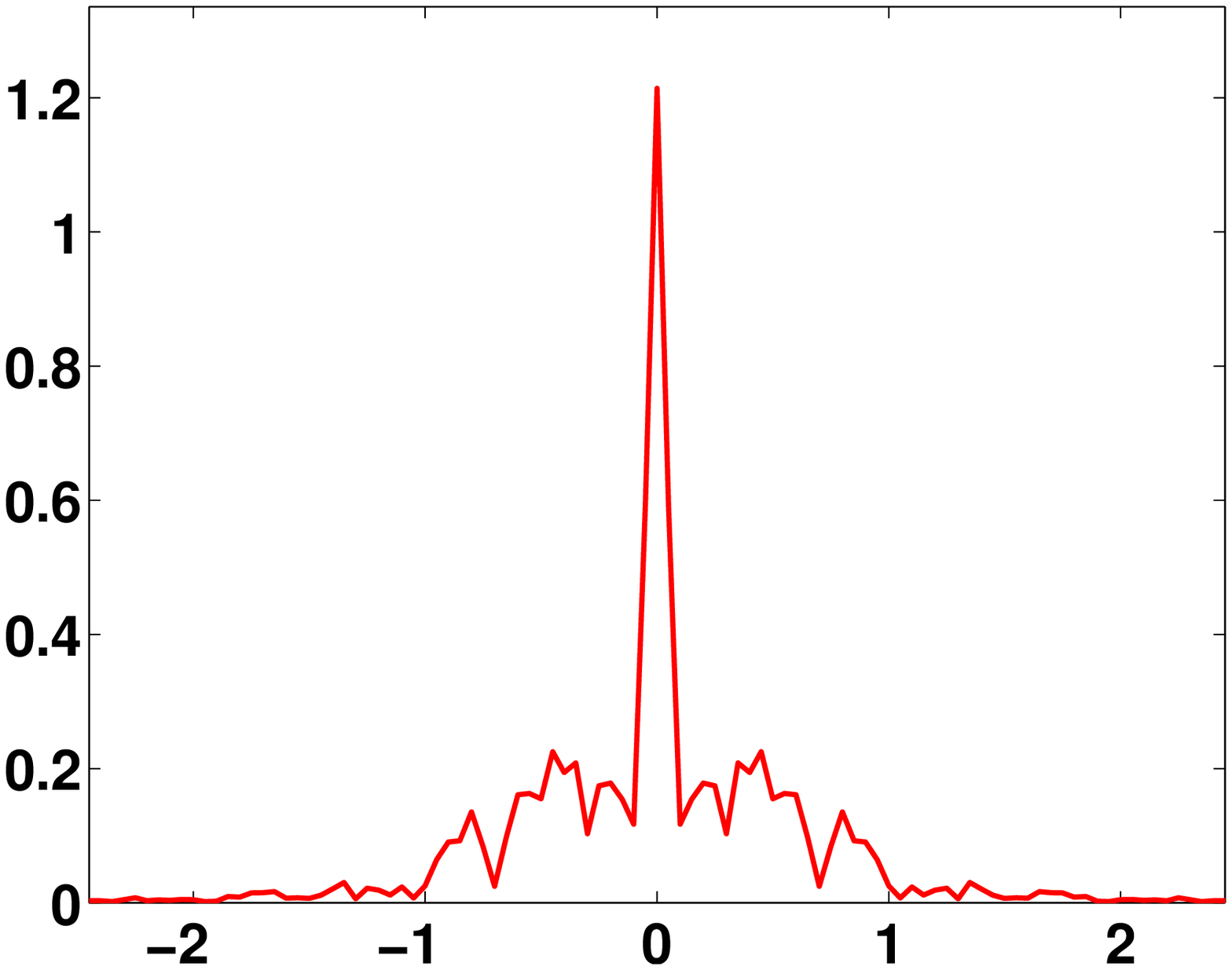}\\
(a)&(b)&(c)&(d)
\end{tabular}
\caption{Illustration of the shift towards low frequencies due to the modulus. (a) $f\star\psi_j$ (real part) (b) $\widehat{f\star\psi_j}$ (in modulus) (c) $|f\star\psi_j|$ (d) $\widehat{|f\star\psi_j|}$. Observe that $\widehat{f\star\psi_j}$ and $\widehat{|f\star\psi_j|}$ are localized on frequency bands of the same width, but that, for $\widehat{|f\star\psi_j|}$, the band is centered around $0$, and thus globally lower in frequency.\label{fig:shift_low}}
\end{figure}

\section{Theorem statement\label{s:exp_decay_thm}}

The theorem of this section relies on similar ideas to the ones of Theorem~\ref{thm:norm_preservation} but improves it in two aspects:
\begin{itemize}
\item It gives a precise bound on $\sum_{p\in\mathcal{P}_J,|p|=n}||U[p]f||_2^2$, instead of simply ensuring that this quantity goes to zero. This bound formalizes the idea described in the end of the previous paragraph, that the energy of the signal carried by the frequencies around $2^j$ disappears after a number of scattering step proportional to $j$.
\item It holds for a much more general class of wavelets that Theorem~\ref{thm:norm_preservation}. In particular, it does not require the wavelets to be analytical. It notably applies to Morlet wavelets, and also to compactly-supported wavelets like Selesnick wavelets \citep{selesnick}, two important choices for practical applications. The wavelets have to satisfy three simple conditions, which will be commented after the statement of the theorem.
\end{itemize}

For any $a>0$, we denote by $\chi_a$ the Gaussian function:
\begin{equation*}
\forall t\in\R\quad\quad
\chi_a(t)=\sqrt{\pi}a\exp\left(-\left(\pi a\right)^2t^2\right)
\end{equation*}
whose Fourier transform satisfies:
\begin{equation*}
\forall\omega\in\R\quad\quad
\hat\chi_a(\omega)=\exp(-(\omega/a)^2)
\end{equation*}

\begin{thm}\label{thm:exp_decay}
Let $(\psi_j)_{j\in\Z}$ be a family of wavelets. We assume the following Littlewood-Paley inequality to be satisfied:
\begin{equation}
\forall \omega\in\R\quad\quad
\frac{1}{2}\underset{j\in \Z}{\sum}
\left(|\hat\psi_j(\omega)|^2+|\hat\psi_j(-\omega)|^2\right)
\leq 1
\label{eq:exp_decay_lp}
\end{equation}
We also assume that:
\begin{equation}\label{eq:smaller_negs}
\forall j\in\Z,\omega> 0\quad\quad
|\hat\psi_j(-\omega)|\leq|\hat\psi_j(\omega)|
\end{equation}
and, for any $\omega$, the inequality is strict for at least one value of $j$.

Finally, we assume that, for some $\epsilon>0$:
\begin{equation}\label{eq:wavelet_in_zero}
\hat\psi(\omega)=O(|\omega|^{1+\epsilon})
\end{equation}
when $\omega\to 0$.

Then, for any $J\in\Z$, there exist $r>0,a>1$ such that, for all $n\geq 2$ and $f\in L^2(\R,\R)$:
\begin{equation}
\underset{p\in\mathcal{P}_J,|p|=n}{\sum}||U[p]f||_2^2\leq
||f||_2^2-||f\star\chi_{ra^n}||_2^2=
 \int_\R|\hat f(\omega)|^2\left(1-|\hat\chi_{ra^n}(\omega)|^2\right)d\omega
\label{eq:exp_decay}
\end{equation}

\end{thm}

\nl
Before turning to the proof of this theorem, let us briefly comment the conditions~\eqref{eq:exp_decay_lp},~\eqref{eq:smaller_negs} and~\eqref{eq:wavelet_in_zero}.

The Littlewood-Paley inequality~\eqref{eq:exp_decay_lp} is equivalent to the fact that the wavelet transform is contractive over $L^2(\R,\R)$. Without a condition of this kind, the wavelet transform can amplify some frequencies of the initial signal, and the energy contained in the long paths of length $n$ will not necessarily decrease when $n$ increases.

The condition~\eqref{eq:smaller_negs} describes the fact that, although they do not need to be analytical, the wavelets must give more weight to the positive frequencies than to the negative ones. If this condition is not met, the phenomenon of energy shift towards the low frequencies may not happen. This is in particular the case if the wavelets are real.

The condition~\eqref{eq:wavelet_in_zero} states that the wavelets have a bit more than one zero momentum. It is necessary for our proof, but it is not clear whether the theorem stays true without it or not.

\section{Principle of the proof\label{s:exp_decay_proof}}

The proof proceeds by iteration over $n$.

\nl
In this paragraph, we show that, if $a$ is close enough to $1$, the fact that the property holds for $n$ implies that it also holds for $n+1$. The initialization, for $n=2$, relies on the same principle, but is more technical; it is done in the paragraph~\ref{ss:exp_decay_init}.

\begin{lem}\label{lem:low_freq_shift}
For any $j\in\Z,x\in\R^*_+,\delta_j\in\R$:
\begin{equation*}
||\,|f\star\psi_j|\star\chi_x||_2^2 \geq ||f\star\psi_j\star(\chi_x e^{2\pi i\delta_j.})||_2^2=\int_\R|\hat f(\omega)|^2|\hat\psi_j(\omega)|^2 |\hat\chi_x(\omega-\delta_j)|^2d\omega
\end{equation*}
\end{lem}
\noindent
This lemma is already present in \citep{group_invariant}. For sake of completeness, we prove it again in the paragraph~\ref{ss:low_freq_shift}.

If the property~\eqref{eq:exp_decay} holds for $n\geq 2$, then:
\begin{align}
\underset{p\in\mathcal{P}_J,|p|=n+1}{\sum}||U[p]f||_2^2
&=\underset{j\leq J}{\sum}\left(\underset{p\in\mathcal{P}_J,|p|=n}{\sum}||U[j,p]f||_2^2\right)\nonumber\\
&=\underset{j\leq J}{\sum}\left(\underset{p\in\mathcal{P}_J,|p|=n}{\sum}||U[p]\big(|f\star\psi_j|\big)\,||_2^2\right)\nonumber\\
&\leq\underset{j\leq J}{\sum}\Big(||\,|f\star\psi_j|\,||_2^2-||\,|f\star\psi_j|\star\chi_{ra^n}||_2^2\Big)\label{eq:exp_decay_hyp_rec}
\end{align}
By Lemma~\ref{lem:low_freq_shift}, for any choice of $(\delta_j)_{j\leq J}$:
\begin{align*}
\underset{p\in\mathcal{P}_J,|p|=n+1}{\sum}||U[p]f||_2^2
&\leq \underset{j\leq J}{\sum}\Big(||f\star\psi_j||_2^2-||f\star\psi_j\star(\chi_{ra^n}e^{2\pi i\delta_j})||_2^2\Big)\\
&=\int_\R|\hat f(\omega)|^2\left(\underset{j\leq J}{\sum}|\hat\psi_j(\omega)|^2(1-|\hat\chi_{ra^n}(\omega-\delta_j)|^2) \right)d\omega
\end{align*}
We symmetrize the expression, by taking into account the fact that $f$ is real, so $|\hat f(\omega)|=|\hat f(-\omega)|$:
\begin{align*}
\underset{p\in\mathcal{P}_J,|p|=n+1}{\sum}||U[p]f||_2^2
\leq
\int_\R|\hat f(\omega)|^2\times\frac{1}{2} \bigg(\underset{j\leq J}{\sum}&|\hat\psi_j(\omega)|^2(1-|\hat\chi_{ra^n}(\omega-\delta_j)|^2) \\
+&|\hat\psi_j(-\omega)|^2(1-|\hat\chi_{ra^n}(-\omega-\delta_j)|^2)\bigg)d\omega
\end{align*}
To conclude, it is thus sufficient to show that, if the $\delta_j$'s are well-chosen, then:
\begin{align*}
\frac{1}{2} \bigg(\underset{j\leq J}{\sum}|\hat\psi_j(\omega)|^2&(1-|\hat\chi_{ra^n}(\omega-\delta_j)|^2)
+|\hat\psi_j(-\omega)|^2(1-|\hat\chi_{ra^n}(-\omega-\delta_j)|^2)\bigg)\\
&\leq 1-|\hat\chi_{a r^{n+1}}(\omega)|^2
\hskip 2cm (\forall\omega\in\R)
\end{align*}
which is a direct consequence of the next lemma, proven in the paragraph~\ref{ss:exp_decay_deltas}.
\begin{lem}\label{lem:exp_decay_deltas}
For any $x>0$, if $a>1$ is close enough to $1$, then there exist $(\delta_j)_{j\leq J}$ such that:
\begin{align*}
\forall \omega\in\R\quad\quad
\frac{1}{2}&\left(\underset{j\in\Z}{\sum}|\hat\psi_j(\omega)|^2(1-|\hat\chi_x(\omega-\delta_j)|^2)\right.\\
&\left.+ \underset{j\in\Z}{\sum}|\hat\psi_j(-\omega)|^2(1-|\hat\chi_x(-\omega-\delta_j)|^2)
\right) \leq 1-|\hat\chi_{a x}(\omega)|^2
\end{align*}
\end{lem}

The proof is summarized in Figure~\ref{fig:schema_exp_decay}.

\begin{figure}
{\def\arraystretch{0.4}
\begin{tabular}{cccccc}
\multicolumn{6}{l}{
\begin{minipage}{\textwidth}
Inductive hypothesis: the energy contained in the paths of length $n+1$ beginning by $j$ is at most the energy of $f\star\psi_j$, multiplied by a high-pass filter (Equation~\eqref{eq:exp_decay_hyp_rec}).
\end{minipage}}\\&&&&&
\\
$\underset{|p|=n}{\sum}||U[j,p]f||_2^2$&$\leq${\Huge $\int$}
\raisebox{-0.5\height}{
\includegraphics[width=0.16\textwidth]{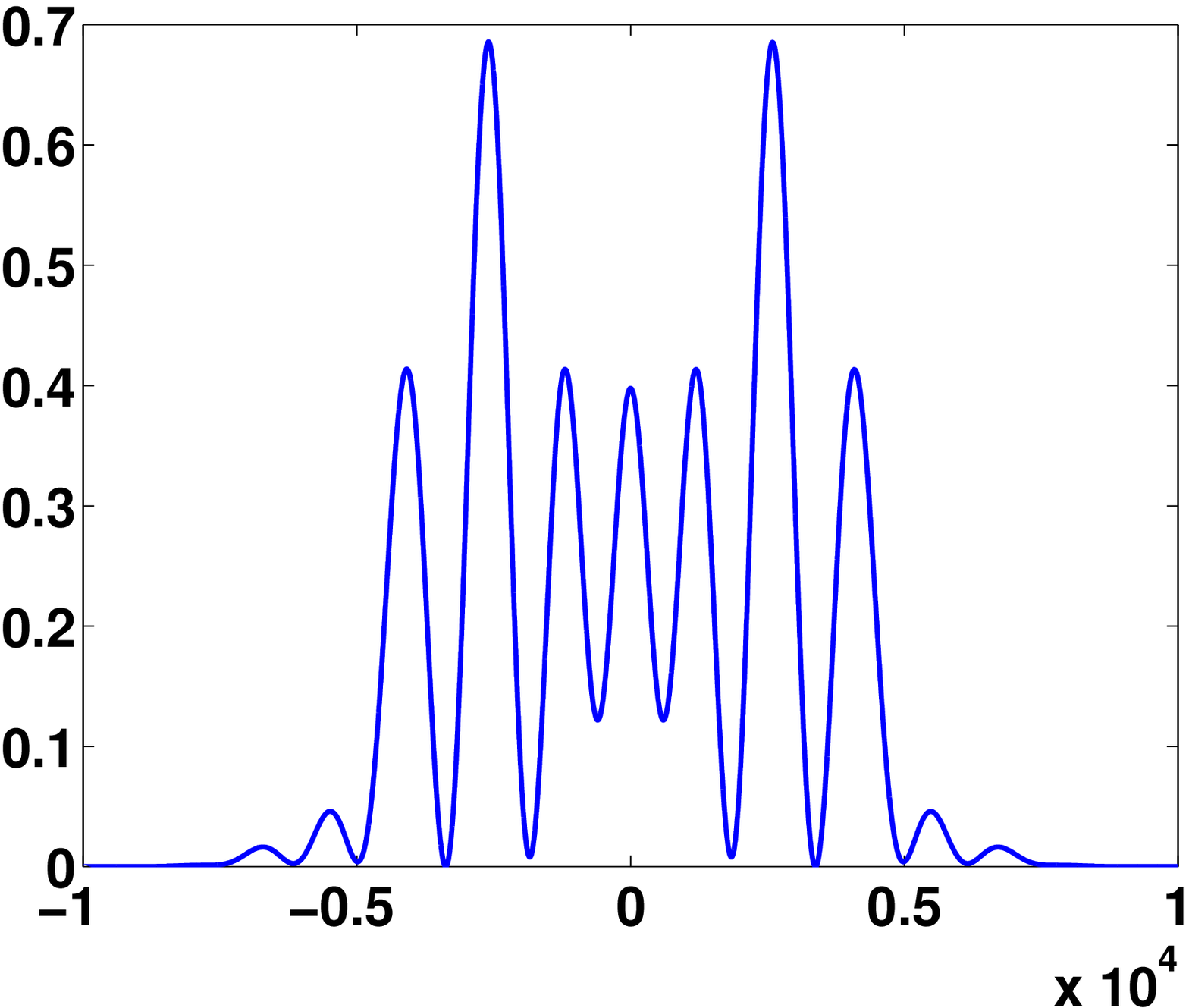}}&
$\times$&
\raisebox{-0.5\height}{
\includegraphics[width=0.16\textwidth]{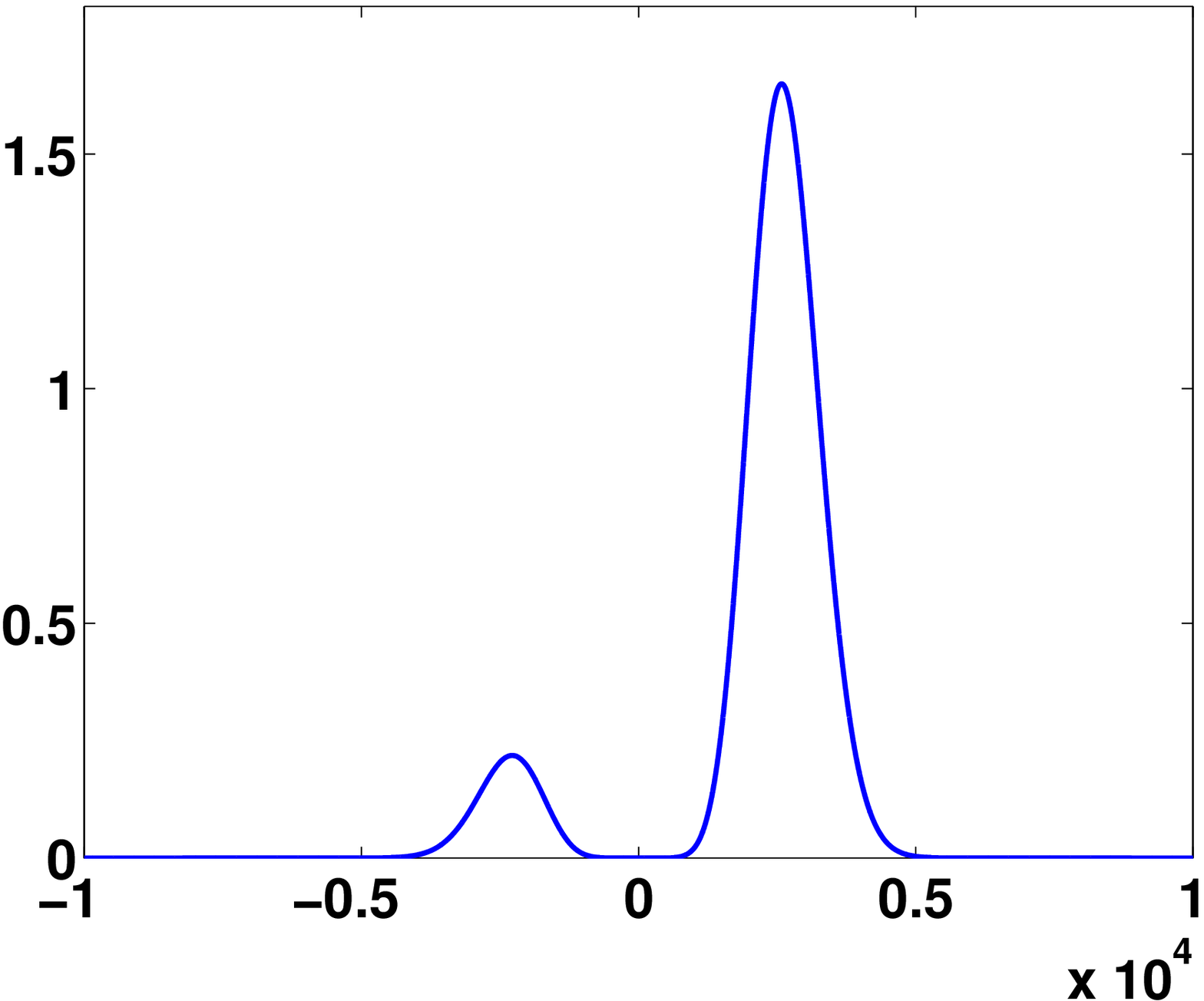}}&
$\times$&\hskip -1cm
\raisebox{-0.5\height}{
\includegraphics[width=0.16\textwidth]{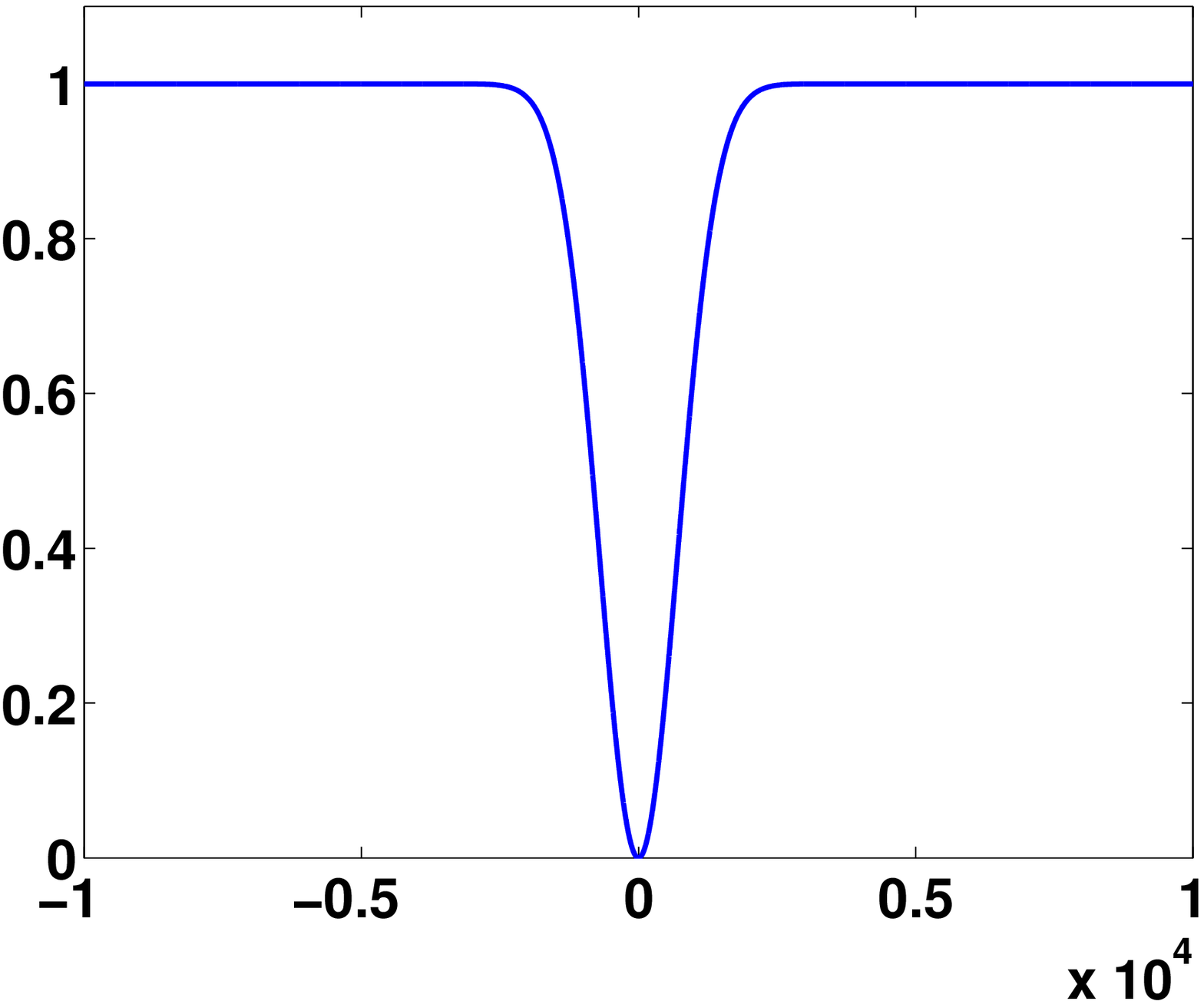}}\\&&&&&\\
&\hskip 1.5cm$|\hat f|^2$&&\hskip 1.2cm$|\hat\psi_j|^2$&&\hskip -0.4cm$1-|\hat\chi_{ar^n}|^2$\\&&&&&\\
\multicolumn{6}{l}{
Because of the modulus around $f\star\psi_j$, the high-pass filter can be arbitrarily shifted in frequency.}\\&&&&&\\
$\underset{|p|=n}{\sum}||U[j,p]f||_2^2$&$\leq${\Huge $\int$}
\raisebox{-0.5\height}{
\includegraphics[width=0.16\textwidth]{exp_decay_ff.eps}}&$\times$&
\raisebox{-0.5\height}{
\includegraphics[width=0.16\textwidth]{exp_decay_psi3.eps}}&
$\times$&\hskip -1cm
\raisebox{-0.5\height}{
\includegraphics[width=0.16\textwidth]{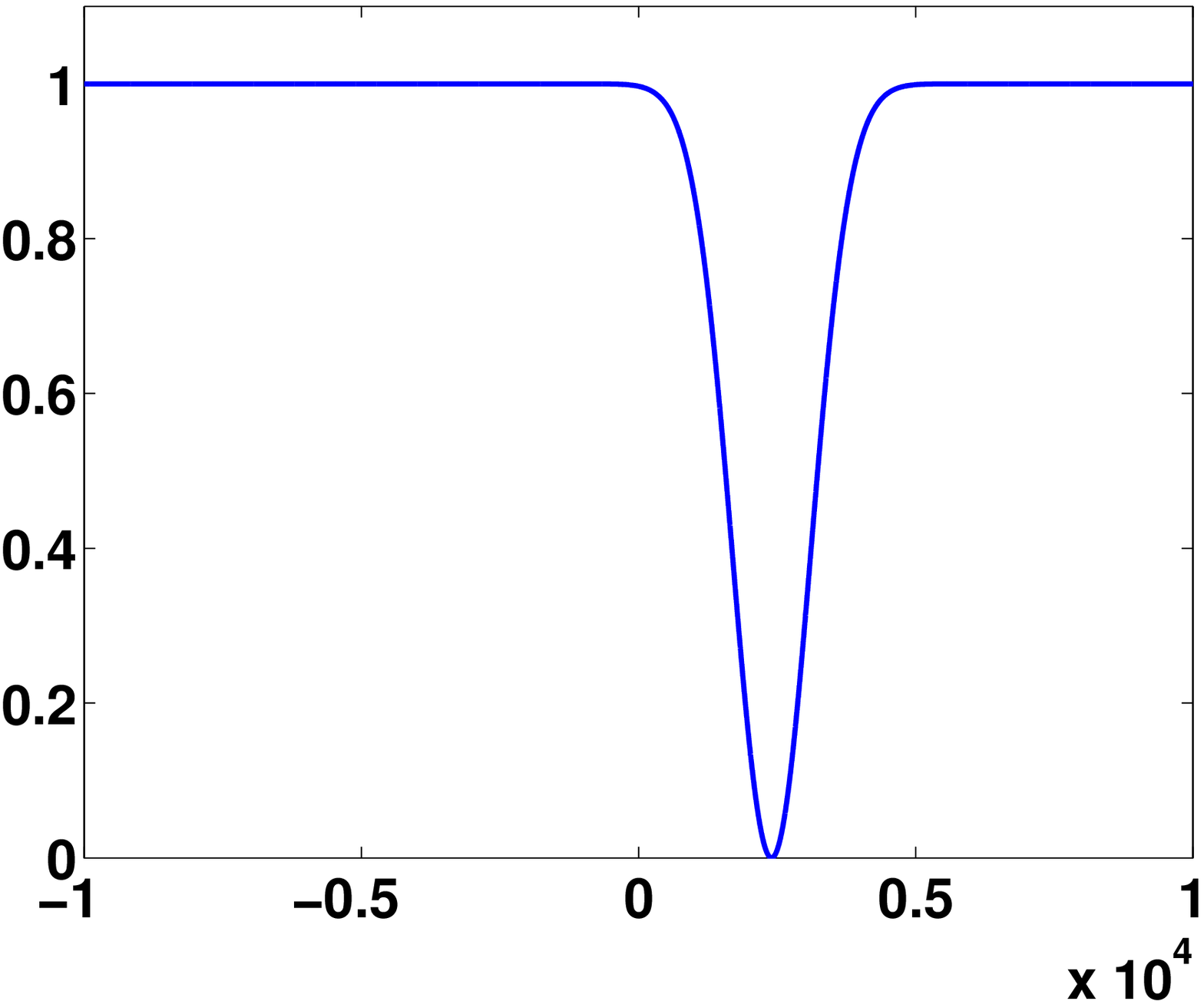}}\\
&\hskip 1.5cm$|\hat f|^2$&&\hskip 1.2cm$|\hat\psi_j|^2$&&\hskip -0.4cm$1-|\hat\chi_{ar^n}(.-\delta_j)|^2$\\&&&&&\\
&$=${\Huge $\int$}
\raisebox{-0.5\height}{
\includegraphics[width=0.16\textwidth]{exp_decay_ff.eps}}&$\times$
&
\multicolumn{3}{c}{
\raisebox{-0.5\height}{
\includegraphics[width=0.16\textwidth]{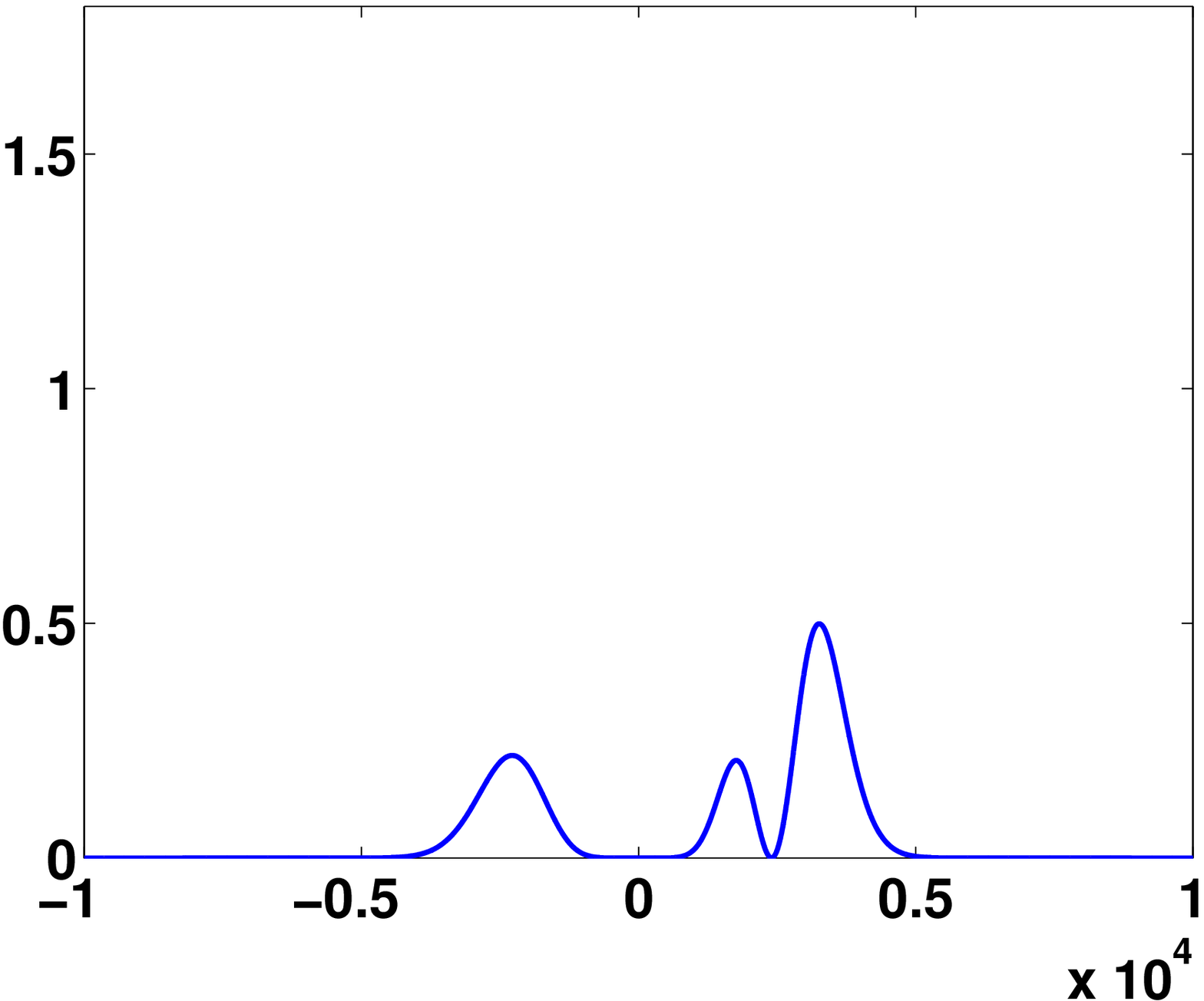}}}\\&&&&
\\
&\hskip 1.5cm$|\hat f|^2$&&\multicolumn{3}{c}{\hskip 1.2cm$|\hat\psi_j|^2(1-|\hat\chi_{ar^n}(.-\delta_j)|^2)$}\\&&&&\\
\multicolumn{6}{l}{
We consider real signals so we can symmetrize the filter.}\\&&&&&\\
&$=${\Huge $\int$}
\raisebox{-0.5\height}{
\includegraphics[width=0.16\textwidth]{exp_decay_ff.eps}}&$\times$
&
\multicolumn{3}{c}{
\raisebox{-0.5\height}{
\includegraphics[width=0.16\textwidth]{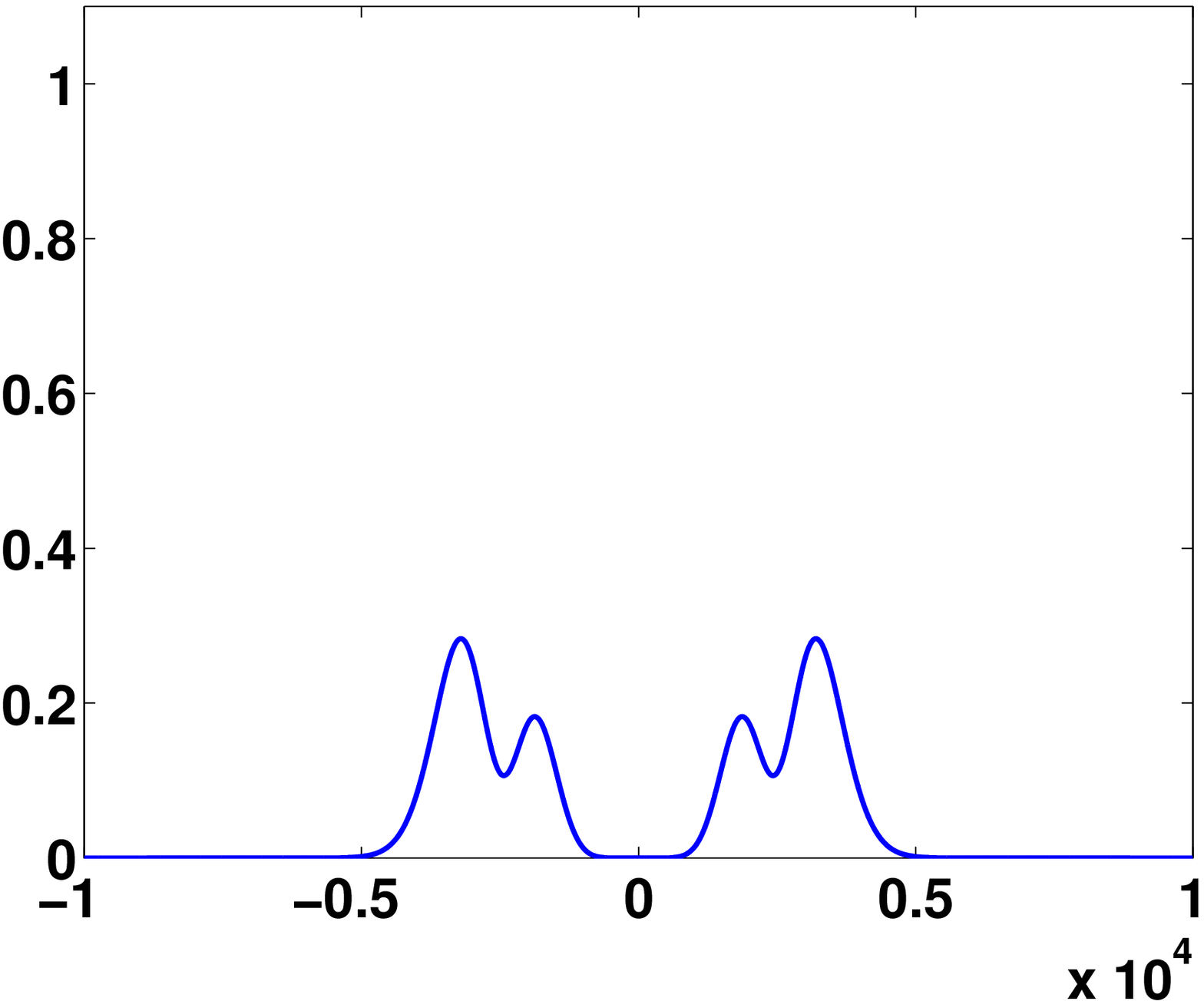}}}\\&&&&&
\\
&\hskip 1.5cm$|\hat f|^2$&&\multicolumn{3}{c}{\footnotesize
\hskip -0.8cm
$\frac{1}{2}\left(|\hat\psi_j|^2(1-|\hat\chi_{ar^n}(.-\delta_j)|^2)
+|\hat\psi_j(-.)|^2(1-|\hat\chi_{ar^n}(-.-\delta_j)|^2)\right)$
}\\
\multicolumn{6}{l}{
Summing over $j$ gives:
}\\&&&&&\\
$\underset{|p|=n+1}{\sum}||U[p]f||_2^2$&$\leq${\Huge $\int$}
\raisebox{-0.5\height}{
\includegraphics[width=0.16\textwidth]{exp_decay_ff.eps}}&$\times$&
\multicolumn{3}{c}{
\raisebox{-0.5\height}{
\includegraphics[width=0.16\textwidth]{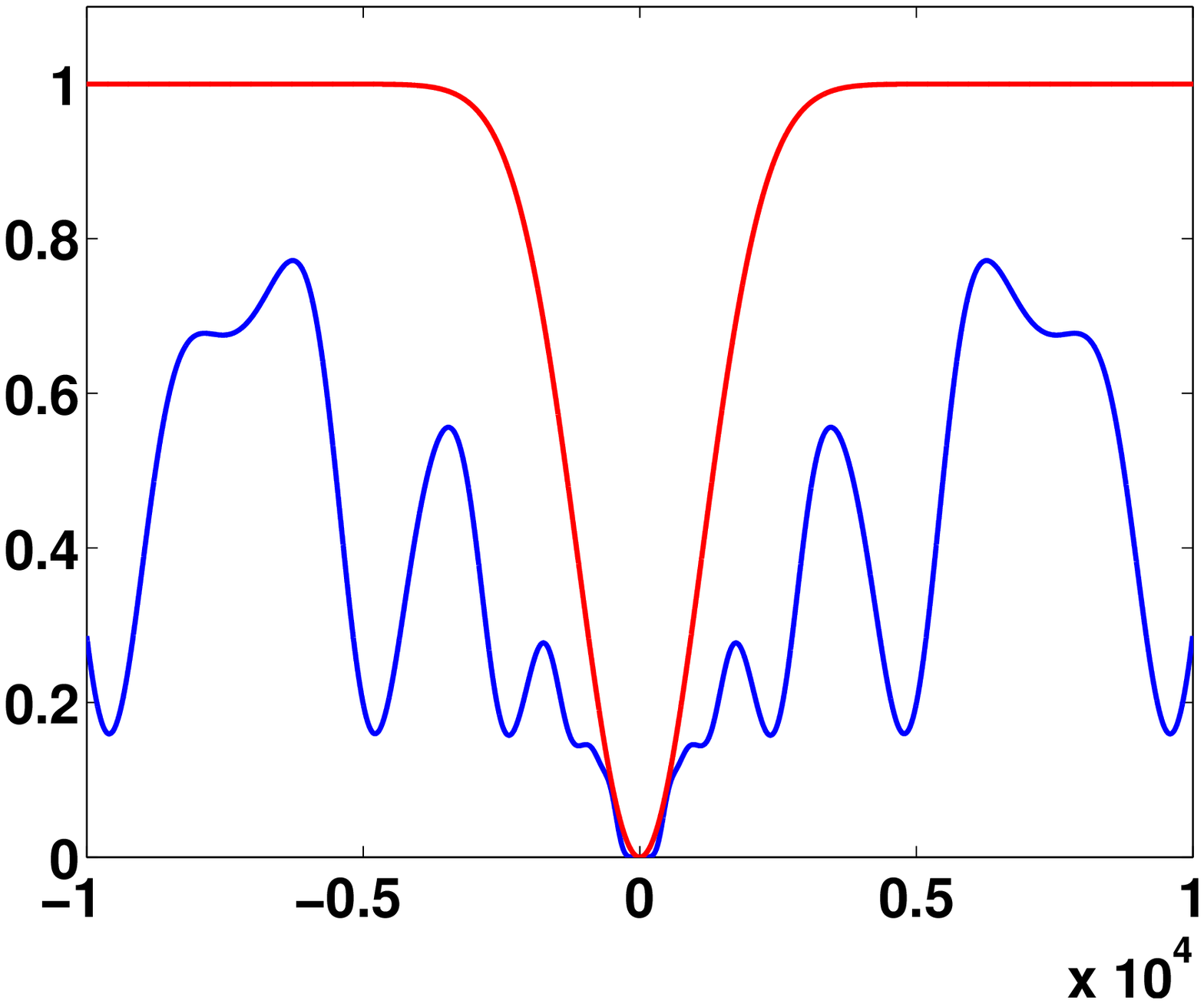}}}\\&&&&&\\
&\hskip 1.5cm$|\hat f|^2$&&\multicolumn{3}{c}{\footnotesize
\hskip -0.8cm
$\underset{j\leq J}{\sum}\frac{1}{2}\left(|\hat\psi_j|^2(1-|\hat\chi_{ar^n}(.-\delta_j)|^2)
+|\hat\psi_j(-.)|^2(1-|\hat\chi_{ar^n}(-.-\delta_j)|^2)\right)$
}\\
&\multicolumn{1}{l}{
$\leq \int$\hskip 1.7cm
$|\hat f|^2$}
&$\times$&
\multicolumn{3}{c}{\hskip 2cm
$1-|\hat\chi_{ar^{n+1}}|^2$ \hskip 1cm(shown in red)}
\end{tabular}}
\caption{Schematic summary of the proof of Theorem~\ref{thm:exp_decay}
\label{fig:schema_exp_decay}}
\end{figure}

\section{Adaptation of the theorem to stationary processes\label{s:exp_decay_adaptation}}

In what follows, $X$ is a real-valued stationary stochastic process, with continuous and integrable autocovariance $R_X$.

The scattering transform of $X$ is defined in the same way as for elements of $L^2(\R)$:
\begin{gather*}
U[\o]X = X\\
\forall n\geq 1, (j_1,...,j_n)\in\Z^n\quad\quad
U[(j_1,...,j_n)]X=|U[(j_1,...,j_{n-1})]X\star \psi_{j_n}|\\
\forall p\in\mathcal{P}_J\quad\quad
S_J[p]f=U[p]f\star\phi_J
\end{gather*}

The results proven for the deterministic wavelet transform tend to be also valid for stationary processes, if one replaces the squared $L^2$-norms by the expectations of the squared modulus. In particular, Theorem~\ref{thm:exp_decay} can be adapted to the case of stationary processes.
\begin{thm}
Let $(\psi_j)_{j\in\Z}$ be a family of wavelets, satisfying the same conditions~\eqref{eq:exp_decay_lp},~\eqref{eq:smaller_negs} and~\eqref{eq:wavelet_in_zero} as in Theorem~\ref{thm:exp_decay}.

Then, for any $J\in\Z$, there exist $r>0,a>1$ such that, for all $n\geq 2$ and $f\in L^2(\R,\R)$:
\begin{equation*}
\underset{p\in\mathcal{P}_J,|p|=n}{\sum}\E(|U[p]X|_2^2)\leq 
\E(|X|^2)-\E(|X\star\chi_{ra^n}|^2)=\int_\R \hat R_X(\omega)(1-|\hat\chi_{ra^n}(\omega)|^2)d\omega
\end{equation*}
\end{thm}

The proof of this theorem is exactly the same as the one of Theorem~\ref{thm:exp_decay}. The only lines to be modified are the ones where we use the fact that the Fourier transform is an isometry of $L^2(\R)$. We use instead the property according to which, for any stationary process with continuous and integrable autocovariance $R_Y$, and for any $h\in L^1\cap L^2(\R)$:
\begin{align*}
\E(|Y\star h|^2)
&=|\E(Y\star h)|^2 +
\int_\R|\hat h(\omega)|^2 \hat R_Y(\omega)d\omega\\
&=\left|\E(Y)\times \int_\R h(s)ds\right|^2 +
\int_\R|\hat h(\omega)|^2 \hat R_Y(\omega)d\omega
\end{align*}

\section{Proof of Lemmas\label{s:exp_decay_technical}}

\subsection{Proof of Lemma~\ref{lem:low_freq_shift}\label{ss:low_freq_shift}}

\begin{lem*}[\ref{lem:low_freq_shift}]
For any $j\in\Z,x\in\R^*_+,\delta_j\in\R$:
\begin{equation*}
||\,|f\star\psi_j|\star\chi_x||_2^2 \geq ||f\star\psi_j\star(\chi_x e^{2\pi i\delta_j.})||_2^2=\int_\R|\hat f(\omega)|^2|\hat\psi_j(\omega)|^2 |\hat\chi_x(\omega-\delta_j)|^2d\omega
\end{equation*}
\end{lem*}
\begin{proof}
For any $t$, because $\chi_x$ is a positive function:
\begin{align*}
|f\star\psi_j|\star\chi_x(t)
&=|(f\star\psi_j)e^{-2\pi i\delta_j.}|\star\chi_x(t)\\
&=|(f\star\psi_j)e^{-2\pi i\delta_j.}|\star|\chi_x|(t)\\
&\geq|((f\star\psi_j)e^{-2\pi i\delta_j.})\star\chi_x(t)|\\
&=|e^{-2\pi i \delta_j t}\big(f\star\psi_j\star(\chi_x\star e^{2\pi i\delta_j.})\big)(t)|\\
&=|f\star\psi_j\star(\chi_x\star e^{2\pi i\delta_j.})(t)|
\end{align*}
This implies the inequality. The equality is a consequence of the unitarity of the Fourier transform.
\end{proof}

\subsection{Proof of Lemma~\ref{lem:exp_decay_deltas}\label{ss:exp_decay_deltas}}

\begin{lem*}[\ref{lem:exp_decay_deltas}]
For any $x>0$, if $a>1$ is close enough to $1$, then there exist $(\delta_j)_{j\leq J}$ such that:
\begin{align*}
\forall \omega\in\R\quad\quad
\frac{1}{2}&\left(\underset{j\in\Z}{\sum}|\hat\psi_j(\omega)|^2(1-|\hat\chi_x(\omega-\delta_j)|^2)\right.\\
&\left.+ \underset{j\in\Z}{\sum}|\hat\psi_j(-\omega)|^2(1-|\hat\chi_x(-\omega-\delta_j)|^2)
\right) \leq 1-|\hat\chi_{a x}(\omega)|^2
\end{align*}
\end{lem*}
\begin{proof}
We are going to look for $\delta_j$'s of the form $\delta_j=\delta 2^{-j}$.
\begin{align*}
&\frac{1}{2}\left(\underset{j\in\Z}{\sum}|\hat\psi_j(\omega)|^2(1-|\hat\chi_x(\omega-\delta 2^{-j})|^2)+ \underset{j\in\Z}{\sum}|\hat\psi_j(-\omega)|^2(1-|\hat\chi_x(-\omega-\delta 2^{-j})|^2)
\right)\\
=&\frac{1}{2}\left(\underset{j\in\Z}{\sum}|\hat\psi_j(\omega)|^2(1-e^{2\left(-\frac{\omega^2}{x^2}+\frac{2\delta\omega 2^{-j}}{x^2}-\frac{\delta^2 2^{-2j}}{x^2}\right)})+ \underset{j\in\Z}{\sum}|\hat\psi_j(-\omega)|^2(1-e^{2\left(-\frac{\omega^2}{x^2}-\frac{2\delta\omega 2^{-j}}{x^2}-\frac{\delta^2 2^{-2j}}{x^2}\right)})
\right)
\end{align*}
Let us define:
\begin{equation*}
\forall \omega\in\R^*\quad\quad S(\omega)=\frac{1}{2}\left(\underset{j\in\Z}{\sum}|\hat\psi_j(\omega)|^2+\underset{j\in\Z}{\sum}|\hat\psi_j(-\omega)|^2\right)
\end{equation*}
This function is never $0$ if $\omega\ne 0$: if $\omega>0$, there exist $j$ such that $|\hat\psi_j(\omega)|>|\hat\psi_j(-\omega)|\geq 0$ so $S(\omega)>0$; as $S$ is even, we also have $S(\omega)>0$ if $\omega<0$.

The function $y\to 1-e^{2\left(-\frac{\omega^2}{x^2}+y\right)}$ is concave so, for any $\omega\ne 0$:
\begin{align}
&\frac{1}{2}\left(\underset{j\in\Z}{\sum}|\hat\psi_j(\omega)|^2(1-
e^{2\left(-\frac{\omega^2}{x^2}+\frac{2\delta\omega 2^{-j}}{x^2}-\frac{\delta^2 2^{-2j}}{x^2}\right)})+ \underset{j\in\Z}{\sum}|\hat\psi_j(-\omega)|^2(1-e^{2\left(-\frac{\omega^2}{x^2}-\frac{2\delta\omega 2^{-j}}{x^2}-\frac{\delta^2 2^{-2j}}{x^2}\right)})
\right)\nonumber\\
&=S(\omega)\underset{j\in\Z}{\sum}\frac{\frac{1}{2}|\hat\psi_j(\omega)|^2}{S(\omega)}(1-e^{2\left(-\frac{\omega^2}{x^2}+\frac{2\delta\omega 2^{-j}}{x^2}-\frac{\delta^2 2^{-2j}}{x^2}\right)})+ \underset{j\in\Z}{\sum}\frac{\frac{1}{2}|\hat\psi_j(-\omega)|^2}{S(\omega)}(1-e^{2\left(-\frac{\omega^2}{x^2}-\frac{2\delta\omega 2^{-j}}{x^2}-\frac{\delta^2 2^{-2j}}{x^2}\right)})\nonumber\\
&\begin{aligned}
\leq S(\omega)\left(1-\exp\left(2\left(-\frac{\omega^2}{x^2}\right.\right.\right.
&+\underset{j\in\Z}{\sum}\frac{\frac{1}{2}|\hat\psi_j(\omega)|^2}{S(\omega)}\left(\frac{2\delta\omega 2^{-j}}{x^2}-\frac{\delta^2 2^{-2j}}{x^2}\right)\\
&\left.\left.\left.
+ \underset{j\in\Z}{\sum}\frac{\frac{1}{2}|\hat\psi_j(-\omega)|^2}{S(\omega)}\left(-\frac{2\delta\omega 2^{-j}}{x^2}-\frac{\delta^2 2^{-2j}}{x^2}\right)
\right)\right)\right)
\end{aligned}\nonumber\\
&=S(\omega)\left(1-\exp\left(2\left(
-\frac{\omega^2}{x^2} + \frac{2\delta\omega}{x^2} F_1(\omega) -\frac{\delta^2}{x^2}F_2(\omega)
\right)\right)\right)\label{eq:S_1_minus_exp}
\end{align}
if we set:
\begin{gather*}
F_1(\omega)=\underset{j\in\Z}{\sum}\left(\frac{1}{2}\frac{|\hat\psi_j(\omega)|^2-|\hat\psi_j(-\omega)|^2}{S(\omega)}\right)2^{-j}\\
F_2(\omega)=\underset{j\in\Z}{\sum}\left(\frac{1}{2}\frac{|\hat\psi_j(\omega)|^2+|\hat\psi_j(-\omega)|^2}{S(\omega)}\right)2^{-2j}
\end{gather*}
The functions $F_1$ and $F_2$ are both continuous. Indeed, for any $j$, $\hat\psi_j$ is continuous, because $\psi\in L^1(\R)$. The function $S$ is also continuous, and lower-bounded by a strictly positive constant. Finally, the sums converge uniformly on every compact subset of $\R$, because of the assumption~\eqref{eq:wavelet_in_zero}. This implies the continuity.

Because of the hypothesis~\eqref{eq:smaller_negs}, $F_1(\omega)>0$ when $\omega>0$. Moreover, for any $\omega>0$, $F_1(2\omega)=2F_1(\omega)$. By a compacity argument, there exist $c>0$ such that:
\begin{equation*}
\forall \omega>0\quad\quad F_1(\omega)\geq c\omega
\end{equation*}
which, as $F_1$ is odd, implies:
\begin{equation}\label{eq:low_bound_F1}
\forall \omega\in\R\quad\quad \omega F_1(\omega)\geq c\omega^2
\end{equation}
Similarly, there exist $C>0$ such that:
\begin{equation}\label{eq:upp_bound_F2}
\forall\omega\in\R\quad\quad F_2(\omega)\leq C\omega^2
\end{equation}
By combining~\eqref{eq:low_bound_F1} and~\eqref{eq:upp_bound_F2} with~\eqref{eq:S_1_minus_exp}, we get that, for all $\omega\ne 0$:
\begin{align*}
&\frac{1}{2}\left(\underset{j\in\Z}{\sum}|\hat\psi_j(\omega)|^2(1-|\hat\chi_x(\omega-\delta 2^{-j})|^2)+ \underset{j\in\Z}{\sum}|\hat\psi_j(-\omega)|^2(1-|\hat\chi_x(-\omega-\delta 2^{-j})|^2)
\right)\\
&\leq S(\omega)\left(1-\exp\left(-2\frac{\omega^2}{x^2}(1-2c\delta+C\delta^2) \right)\right)
\end{align*}
If we take $\delta = c/C$, this yields, for any $a$ such that $1<a\leq \frac{1}{\sqrt{1-c^2/C}}$:
\begin{align*}
&\frac{1}{2}\left(\underset{j\in\Z}{\sum}|\hat\psi_j(\omega)|^2(1-|\hat\chi_x(\omega-\delta 2^{-j})|^2)+ \underset{j\in\Z}{\sum}|\hat\psi_j(-\omega)|^2(1-|\hat\chi_x(-\omega-\delta 2^{-j})|^2)
\right)\\
&\leq S(\omega)\left(1-\exp\left(-2\frac{\omega^2}{x^2}\left(1-\frac{c^2}{C}\right)\right)\right)\\
&\leq S(\omega)\left(1-\exp\left(-2\frac{\omega^2}{(ax)^2}\right)\right)\\
&\leq 1-\exp\left(-2\frac{\omega^2}{(ax)^2}\right)\\
&= 1- |\hat\chi_{ax}(\omega)|^2
\end{align*}
The bound $S(\omega)\leq 1$ comes from the Littlewood-Paley inequality~\eqref{eq:exp_decay_lp}.

The inequalities are also true for $\omega=0$ because all terms are then equal to zero.
\end{proof}

\subsection{Initialization\label{ss:exp_decay_init}}

In this paragraph, we prove that Theorem~\ref{thm:exp_decay} holds for $n=2$. More precisely, we prove that, for any $a>1$, there exist $r>0$ such that Equation~\eqref{eq:exp_decay} is valid for $n=2$.

\begin{proof}
For any real-valued function $g\in L^2(\R)$:
\begin{align}
\underset{j\leq J}{\sum}||g\star\psi_j||_2^2
&=\int_\R|\hat g(\omega)|^2\left(\underset{j\leq J}{\sum}|\hat\psi_j(\omega)|^2\right)d\omega\nonumber\\
&=\int_\R|\hat g(\omega)|^2\left(\frac{1}{2}\underset{j\leq J}{\sum}|\hat\psi_j(\omega)|^2+|\hat\psi_j(-\omega)|^2\right)d\omega\label{eq:bound_j_leq_J}
\end{align}
If the following inequality held, for some $r>0,a>1$:
\begin{equation*}
\frac{1}{2}\underset{j\leq J}{\sum}\left(|\hat\psi_j(\omega)|^2+|\hat\psi_j(-\omega)|^2\right)\leq 1-|\hat\chi_{ra}(\omega)|^2
\end{equation*}
then Theorem~\ref{thm:exp_decay} would be valid for $n=1$ and it could be proven for $n=2$ in the exact same way as in Section~\ref{s:exp_decay_proof}. Unfortunately, this inequality is not necessarily valid (in particular, it is never the case if the wavelet transform is unitary and the wavelets are band-limited).

The next lemma (proven at the end of the paragraph) nevertheless shows that the inequality is satisfied, if one allows the Gaussian function to be replaced by a more general function.
\begin{lem}\label{lem:positive_low_pass}
There exist a real-valued positive function $\phi\in L^1\cap L^2(\R)$ such that:
\begin{equation}
|\hat\phi(\omega)|^2=1-O(\omega^2)\quad\quad\mbox{when }\omega\to 0
\label{eq:pos_low_pass_zero}
\end{equation}
and:
\begin{equation}
\forall \omega\in\R\quad\quad \frac{1}{2}\underset{j\leq J}{\sum}\left(|\hat\psi_j(\omega)|^2+|\hat\psi_j(-\omega)|^2\right)\leq 1-|\hat\phi(\omega)|^2
\label{eq:pos_low_pass_lp}
\end{equation}
\end{lem}
The rest of the proof consists in showing how to adapt Section~\ref{s:exp_decay_proof}, when the Gaussian function has been replaced by the $\phi$ of the lemma.

Because of~\eqref{eq:bound_j_leq_J}, with $\phi$ defined as in Lemma~\ref{lem:positive_low_pass}, we have, for any real-valued function $g\in L^2(\R)$:
\begin{align*}
\underset{j\leq J}{\sum}||g\star\psi_j||_2^2&\leq \int_\R|\hat g(\omega)|^2(1-|\hat\phi(\omega)|^2)\\
&=||g||_2^2-||g\star\phi||_2^2
\end{align*}
So:
\begin{align}
\underset{p\in\mathcal{P}_J,|p|=2}{\sum}||U[p]f||_2^2
&=\underset{j_1\leq J}{\sum}\,\underset{j_2\leq J}{\sum}||\,||f\star\psi_{j_1}|\star\psi_{j_2}|\,||_2^2\nonumber\\
&\leq\underset{j\leq J}{\sum}||f\star\psi_j||_2^2-||\,|f\star\psi_j|\star\phi||_2^2\label{eq:exp_decay_init_eq1}
\end{align}
Lemma~\ref{lem:low_freq_shift} is still true when $\chi_x$ is replaced by $\phi$, because the only property of $\chi_x$ which is really needed is its positivity. So for any $\delta\in\R$:
\begin{equation*}
\forall j\in\Z\quad\quad
||\,|f\star\psi_j|\star\phi||_2^2
\geq \int_\R|\hat f(\omega)|^2|\hat\psi_j(\omega)|^2|\hat\phi(\omega-\delta)|^2d\omega
\end{equation*}
In Section~\ref{s:exp_decay_proof}, we used this same inequality, with a different $\delta$ for each value of $j$. Here, we do not need $\delta$ to vary as a function of $j$; however, we will need to consider different values of $\delta$ and to average the inequalities over these different values.

If we combine the last inequality with~\eqref{eq:exp_decay_init_eq1}, we obtain:
\begin{equation*}
\forall \delta\in\R\quad\quad
\underset{p\in\mathcal{P}_J,|p|=2}{\sum}||U[p]f||_2^2
\leq \int_\R |\hat f(\omega)|^2(1-|\hat\phi(\omega-\delta)|^2) \left(\underset{j\leq J}{\sum}|\hat\psi_j(\omega)|^2\right)d\omega
\end{equation*}
As it holds for any $\delta\in\R$, we have, for any positive $c\in L^1(\R)$ whose integral over $\R$ is $1$:
\begin{align*}
\underset{p\in\mathcal{P}_J,|p|=2}{\sum}||U[p]f||_2^2
&\leq \int_\R c(\delta) \int_\R |\hat f(\omega)|^2(1-|\hat\phi(\omega-\delta)|^2) \left(\underset{j\leq J}{\sum}|\hat\psi_j(\omega)|^2\right)d\omega d\delta\\
&\leq \int_\R |\hat f(\omega)|^2(1-|\hat\phi|^2\star c(\omega)) \left(\underset{j\leq J}{\sum}|\hat\psi_j(\omega)|^2\right)d\omega
\end{align*}
We limit ourselves to the case where $c$ is even. As $|\hat\phi|^2$ is also even, the last inequality yields, by using the fact that $f$ is real:
\begin{equation*}
\underset{p\in\mathcal{P}_J,|p|=2}{\sum}||U[p]f||_2^2
\leq \int_\R |\hat f(\omega)|^2(1-|\hat\phi|^2\star c(\omega))\times \frac{1}{2} \left(\underset{j\leq J}{\sum}|\hat\psi_j(\omega)|^2+|\hat\psi_j(-\omega)|^2\right)d\omega
\end{equation*}
The conclusion comes from a last lemma, proven at the end of the paragraph.
\begin{lem}\label{lem:gaussian_bound}
If we take $c(\delta)=\frac{1}{\sqrt{\pi}}\exp(-\delta^2)$, then there exists $x>0$ such that:
\begin{equation*}
\forall \omega\in\R\quad\quad
(1-|\hat\phi|^2\star c(\omega))\times \frac{1}{2} \left(\underset{j\leq J}{\sum}|\hat\psi_j(\omega)|^2+|\hat\psi_j(-\omega)|^2\right)
\leq 1-|\hat\chi_x(\omega)|^2
\end{equation*}
\end{lem}

\end{proof}

We now give the proofs of Lemmas~\ref{lem:positive_low_pass} and~\ref{lem:gaussian_bound}.
\begin{proof}[Proof of Lemma~\ref{lem:positive_low_pass}]
Let $\gamma:\R\to\R$ be any even, rapidly decreasing and band-limited function such that $\int_\R\gamma^2=1$. We set $\phi_0=\gamma^2$. This is also an even, rapidly decreasing and band-limited function.

Because $\int_\R\gamma^2=1$, we have $\hat\phi_0(0)=1$. As $\phi_0$ is even, $\hat\phi_0'(0)=0$. But the second derivative is strictly negative: $\hat\phi_0''(0)=-(2\pi)^2\int_\R t^2\phi_0(t)dt<0$. We deduce from these relations that there exists $\alpha>0$ such that:
\begin{equation}\label{eq:pos_low_pass_dl0}
|\hat\phi_0(\omega)|^2=1-\alpha\omega^2+o(\omega^2)\quad\quad\mbox{when }\omega\to 0
\end{equation}
 Let us show that, for $M$ large enough, $\phi:t\to M^{-1}\phi_0(M^{-1}t)$ satisfies the desired properties. By construction, it is a real-valued positive function. It is rapidly decreasing, so $\phi\in L^1\cap L^2(\R)$. The property~\eqref{eq:pos_low_pass_zero} holds; only the inequality~\eqref{eq:pos_low_pass_lp} is left to prove.

Because of Equation~\eqref{eq:pos_low_pass_dl0} and because $\hat\phi_0$ is compactly-supported, there exists $\tilde\alpha>0$ such that for all $\omega\in\mbox{Supp}(\hat\phi_0)$:
\begin{equation*}
|\hat\phi_0(\omega)|^2\leq 1-\tilde\alpha\omega^2
\end{equation*}
By the assumption~\eqref{eq:wavelet_in_zero}, $|\hat\psi(\omega)|=O(|\omega|^{1+\epsilon})$ when $\omega\to 0$, for some $\epsilon>0$. This implies that:
\begin{equation*}
\frac{1}{2}\underset{j\leq J}{\sum}\left(|\hat\psi_j(\omega)|^2+|\hat\psi_j(-\omega)|^2\right)=o(\omega^2)
\quad\quad\mbox{when }\omega\to 0
\end{equation*}
Because this sum is moreover bounded by $1$ on all $\R$, there exists $A>0$ such that:
\begin{equation*}
\forall \omega\in\R\quad\quad
\frac{1}{2}\underset{j\leq J}{\sum}\left(|\hat\psi_j(\omega)|^2+|\hat\psi_j(-\omega)|^2\right)\leq A\omega^2
\end{equation*}
If $M\geq \sqrt{A/\tilde\alpha}$, then, on the support of $\hat\phi$:
\begin{align*}
&|\hat\phi(\omega)|^2+\frac{1}{2}\underset{j\leq J}{\sum}\left(|\hat\psi_j(\omega)|^2+|\hat\psi_j(-\omega)|^2\right)\\
&=|\hat\phi_0(M\omega)|^2+\frac{1}{2}\underset{j\leq J}{\sum}\left(|\hat\psi_j(\omega)|^2+|\hat\psi_j(-\omega)|^2\right)\\
&\leq 1-\tilde\alpha M^2\omega^2+A\omega^2\\
&\leq 1
\end{align*}
Outside the support of $\hat\phi$, the inequality is also true, because of the Littlewood-Paley condition~\eqref{eq:exp_decay_lp}. So Equation~\eqref{eq:pos_low_pass_lp} holds.
\end{proof}

\begin{proof}[Proof of Lemma~\ref{lem:gaussian_bound}]
Let us define:
\begin{equation*}
F(\omega)=(1-|\hat\phi|^2\star c(\omega))\times \frac{1}{2} \left(\underset{j\leq J}{\sum}|\hat\psi_j(\omega)|^2+|\hat\psi_j(-\omega)|^2\right)
\end{equation*}
We are going to prove that $F$ has the following three properties:
\begin{enumerate}
\item $\forall \omega\in\R,F(\omega)<1$
\item $F(\omega)=O(\omega^2)$ when $\omega\to 0$
\item There exists $x>0$ such that $F(\omega)\leq 1-|\hat \chi_x(\omega)|^2$ if $|\omega|$ is large enough.
\end{enumerate}
These three properties imply that $F$ is bounded by $1-|\hat\chi_x|^2$ on all $\R$, for $x$ small enough. This assertion relies on a compacity argument; as it is relatively straightforward, we do not prove it.

The first property is an immediate consequence of the Littlewood-Paley inequality~\eqref{eq:exp_decay_lp} and of the fact that $|\hat\phi|^2\star c>0$ over $\R$.

The second one is a consequence of a fact that has been explained in the proof of Lemma~\ref{lem:positive_low_pass}:
\begin{equation*}
\frac{1}{2}\left(\underset{j\leq J}{\sum}|\hat\psi_j(\omega)|^2+|\hat\psi_j(-\omega)|^2\right)=o(\omega^2)\quad\quad\mbox{when }\omega\to 0
\end{equation*}

For the last one, we remark that, for any $\omega\in\R$ such that $\omega \geq 1$:
\begin{align*}
|\hat\phi|^2\star c(\omega)
&=\int_\R|\hat\phi(\delta)|^2 c(\omega-\delta)d\delta\\
&\geq \int_0^1|\hat\phi(\delta)|^2c(\omega-\delta)d\delta\\
&\geq c(\omega) \int_0^1|\hat\phi(\delta)|^2d\delta
\end{align*}
As all the functions are even, this is also true for $\omega\leq-1$. So when $|\omega|$ is large enough:
\begin{equation*}
|\hat\phi|^2\star c(\omega)
\geq \frac{1}{\sqrt{\pi}}\exp(-\omega^2)\left(\int_0^1|\hat\phi(\delta)|^2d\delta\right)
\geq \exp(-2\omega^2)
=|\hat\chi_1(\omega)|^2
\end{equation*}
This proves the third property and concludes.
\end{proof}

\bibliographystyle{plainnat}
\bibliography{../bib_articles.bib,../bib_proceedings.bib,../bib_livres.bib,../bib_misc.bib}

\end{document}

%% file: defs.tex
\newcommand{\nl}{\vskip 0.1cm}
\newcommand{\R}{\mathbb{R}}
\newcommand{\E}{\mathbb{E}}

\newcommand{\N}{\mathbb{N}}
\newcommand{\Z}{\mathbb{Z}}

\newcounter{Exocount}
\newcounter{Exocount_1}
\newcounter{Questcount}
\newcounter{Ssquestcount}
\newtheorem{thm}{Theorem}[section]
\newtheorem{lem}[thm]{Lemma}
\newtheorem*{lem*}{Lemma}

%% file: tikz_scattering.tex
\begin{tikzpicture}[scale=0.88,->,>=stealth',
  level distance=3cm,
  level 1/.style={sibling distance=3cm}
  ]
  \node [] (lev1_r){\footnotesize$f$}
    child {node []{}}
    child {node [draw]{\footnotesize$|f\star\psi_{J-2}|$}
      [level distance=1.5cm,sibling distance=1.5cm]
      child {node []{}}
      child{node[]{}}
      child {node []{}}}
    child {node [draw]{\footnotesize$|f\star\psi_{J-1}|$}
      [level distance=1.5cm,sibling distance=1.5cm]
      child {node []{}}
      child{node []{}}
      child {node []{}}}
    child {node [draw]{\footnotesize$|f\star\psi_{J}|$}
      [level distance=3cm,sibling distance=3cm]
      child {node []{}}
      child{node [draw]{\footnotesize$||f\star\psi_{J}|\star\psi_{J-1}|$}}
      child{node [draw]{\footnotesize$||f\star\psi_{J}|\star\psi_{J}|$}
          [level distance=2cm,sibling distance=2cm]
          child{node[]{}}
          child{node[]{}}
          child{node[]{}}}};

  \node[left] at (-6,-0.2) {\footnotesize$f\star\phi_J$};

  \node[left] at (-6,-2.2) {\footnotesize$...$};
  \node[left] at (-6,-2.7) {\footnotesize$|f\star\psi_{J-2}|\star\phi_J$};
  \node[left] at (-6,-3.2) {\footnotesize$|f\star\psi_{J-1}|\star\phi_J$};
  \node[left] at (-6,-3.7) {\footnotesize$|f\star\psi_{J}|\star\phi_J$};

  \node[left] at (-6,-5.2) {\footnotesize$...$};
  \node[left] at (-6,-5.7) {\footnotesize$||f\star\psi_{J-1}|\star\psi_J|\star\phi_J$};
  \node[left] at (-6,-6.2) {\footnotesize$...$};
  \node[left] at (-6,-6.7) {\footnotesize$||f\star\psi_{J}|\star\psi_{J-1}|\star\phi_J$};
  \node[left] at (-6,-7.2) {\footnotesize$||f\star\psi_{J}|\star\psi_J|\star\phi_J$};

  \node[] at (-4.4,-3.2) {\footnotesize$...$};
  \node[] at (-1.5,-4.8) {\footnotesize$...$};
  \node[] at (1.5,-4.8) {\footnotesize$...$};
  \node[] at (1.6,-6.2) {\footnotesize$...$};
  \node[] at (7.5,-8.2) {\footnotesize$...$};

  \node[] (lev1_r) at (-4.8,-0.2) {};
  \node[] (lev1) at (-5.8,-0.2) {};
  \node[] (lev2_r) at (-4.8,-3.2) {};
  \node[] (lev3_r) at (-4.8,-6.2) {};
  \node[] (lev2) at (-5.8,-3.2) {};
  \node[] (lev3) at (-5.8,-6.2) {};
  \draw[dashed] (lev1_r) -- (lev1);
  \draw[dashed] (lev2_r) -- (lev2);
  \draw[dashed] (lev3_r) -- (lev3);

  \node[left] at (-4.9,-8.2) {\footnotesize$...$};
\end{tikzpicture}